\pdfoutput=1
\documentclass[11pt]{article}
\usepackage{vmargin}
\usepackage{times} 
\usepackage{ifpdf}

\usepackage[T1]{fontenc}
\usepackage{epsfig}
\usepackage{graphicx}
\usepackage{amssymb}
\usepackage{amsmath}
\usepackage{mathrsfs}
\usepackage{latexsym} 
\usepackage[utf8]{inputenc}
\usepackage{amsfonts}
\usepackage{amsthm}
\usepackage{paralist}
\usepackage{stmaryrd}

\usepackage[colorlinks=true]{hyperref}

\usepackage{microtype}

\theoremstyle{plain}
\newtheorem{corollary}{Corollary}
\newtheorem{theorem}[corollary]{Theorem}
\newtheorem{lemma}[corollary]{Lemma}
\newtheorem{proposition}[corollary]{Proposition}
\newtheorem*{theorem*}{Theorem}
\newtheorem*{lemma*}{Lemma}
\newtheorem*{definition*}{Definition}
\newtheorem*{corollary*}{Corollary}

\theoremstyle{definition}

\theoremstyle{remark}

\newcommand{\R}{\mathbb{R}}
\newcommand{\Z}{\mathbb{Z}}
\newcommand{\proba}{\mathbb{P}}
\newcommand{\N}{\mathbb{N}}
\newcommand{\E}{\mathbb{E}}

\renewcommand{\L}{\Lambda}

\newcommand{\Uf}{\mathfrak{U}}
\newcommand{\Xf}{\mathfrak{X}}
\newcommand{\Xc}{\mathcal{X}}
\newcommand{\Uc}{\mathcal{U}}
\newcommand{\Xb}{\mathbf{X}}

\newcommand{\1}{\mathbf{1}}

\DeclareMathOperator*{\limit}{\longrightarrow}
\DeclareMathOperator*{\asym}{\sim}
\DeclareMathOperator*{\approxi}{\approx}

\newcommand*{\TV}{\text{TV}}

\begin{document}
\title{A.S. convergence for infinite colour P\'olya urns associated with stable random walks}
\author{Arthur Blanc-Renaudie\thanks{Universit\'e Paris Saclay, Orsay, ablancrenaudiepro@gmail.com}}
\date{\today}
\maketitle
\begin{abstract} 
We answer Problem 11.1 of Janson \cite{JansonWalkUrn} on P\'olya urns associated with stable random walk. Our proof use neither martingales nor trees, but an approximation with a differential equation.
\end{abstract}
\section{Introduction}
\subsection{Model and motivations}
We consider \emph{single ball addition random walk (SBARW) P\'olya urns} defined as follows. Let $d\in \N$. Let $\Delta$ be a Borel random variable on $\R^d$. Let $(\Delta_i)_{i\in \N}$ be independent copies of $\Delta$. Let $X_1=\Delta_1$. Then, independently for every $n\geq 2$, let $U_n$ be uniform in $\{1,\dots, n-1\}$ and let $X_n=X_{U_{n}}+\Delta_n$. We are interested in the random measure as $n\to \infty$
\[ \mu_{n} := \sum_{i=1}^n \delta_{X_i}. \]

Morally, we give colors $(X_i)_{i\in \N}$ to the balls of a classical P\'olya Urn, which corresponds to positions of some random walks with displacement $\Delta$. Since for $n\in \N$, $X_n$ corresponds to the position of such a random walk after approximatively $\log n$ steps, it is not hard to show, under natural assumptions for $\Delta$, that $X_n$ after proper renormalisation converge in law. Our problem is: Can we say the same thing for $\mu_n$? We show an almost sure (a.s.) convergence when $\Delta$ is in the stable regime, which was conjectured by Janson in  \cite{JansonWalkUrn} Problem 11.1.

This model of Urn was first studied by Bandyopadhyay and Thacker \cite{Bandyb,Bandyc,Bandya} who first studied precisely the law of $(X_n)$ and then showed the convergence of $\mu_n$ in distribution. At the same period, Mailler and Marckert \cite{UrnMarckertMailler}, also proved this convergence for $\mu_n$, and proved an almost sure convergence when $\Delta$ have finite exponential moment. Later, Janson \cite{JansonWalkUrn} extended this almost sure convergence assuming only a second moment. 

All those authors also studied another model of urn: the \emph{deterministic addition random walk (DARW) P\'olya urns} which is similar to the SBARW model. Although our method can be used without much change to study both models, we will focus for simplicity on the SBARW model.
\paragraph{Acknowledgment}  I was supported by the ERC consolidator grant 101001124 (UniversalMap), as well as ISF grants 1294/19 and 898/23. Also, I am thankful to two anonymous referee for their many detailed remarks on an earlier version of this paper.

\subsection{Main result and differential equation}
In this paper we assume that $\Delta$ is in the stable regime. That is we assume that there exists $\alpha>0$, and a random variable $\L$ on $\R^d$ of law $\mathcal L(\Lambda)$ such that as $n\to \infty$,
\begin{equation} \frac{1}{n^\alpha} \sum_{i=1}^n  \Delta_i \limit^{(d)} \L . \label{eq:CVloi} \end{equation}

Those random variables are well-known (see e.g. Kolmogorov and Gnedenko \cite{KolmogorovGnedenko} Chapter 7) and notably there exists $\beta>0$ such that as $n\to \infty$, 
\begin{equation} \proba(\|\Delta\|>n^{\beta})=o(1/n). \label{eq:tail} \end{equation}

For every $a>0$, and finite Borel measure $\mu$ on $\R^d$ we define its renormalization $\Theta_a(\mu)$ as the unique Borel measure such that for every Borel set $B\subset \R^d$ we have $\Theta_a(\mu)(B/a)=\mu(B)$. 

The main goal of this paper is to show the following theorem.
\begin{theorem}\label{Thm} Under \eqref{eq:CVloi}, almost surely,  we have the following weak convergence as $n\to \infty$,
\[ \Theta_{\log(n)^{\alpha}}(\mu_n/n)\limit \mathcal L(\Lambda) .\]
\end{theorem}

We use a completely different method from the ones used in \cite{Bandyb,Bandyc,Bandya, JansonWalkUrn, UrnMarckertMailler}, which are based on martingales, Yule trees and branching random walk. Our proof use none of those tools, and is instead based on an approximation with a differential equation for $(\mu_n)_{n\in \N}$.

We can think of our urn problem as the following informal approximative differential equation:
\begin{equation} \frac{\mathrm{d}}{\mathrm{d}t} \mu_t \approx \frac{1}{t} \mu_t\ast \Delta, \label{9/4/17ha}\end{equation}
where $\ast$ denotes the convolution between two measures. To make use of this differential equation, we consider some times $0=T_0<T_1<T_2<\dots <T_n<\dots$, and use for $n\in \N$,
\[ \mu_{T_n}=\sum_{i=1}^n (\mu_{T_i}-\mu_{T_{i-1}}).\]
This allows us to split the complex estimate of $(\mu_n)_{n\in \N}$ into many simpler ones. To this end, we follow \eqref{9/4/17ha}, and show that approximatively for $i\in \N$, 
\begin{equation} \mu_{T_i}-\mu_{T_{i-1}} \approx (T_i-T_{i-1}) \frac{1}{T_{i-1}} \mu_{T_{i-1}}\ast \Delta . \label{9/4/17hb}\end{equation}
We will detail how we prove this approximation in the next section.

 For this approach to work, $(T_i)_{i\in \N}$ must be properly chosen. Indeed, the faster $(T_i)_{i\in \N}$ grows, the harder it is to study $\mu_n(T_i)-\mu_n(T_{i-1})$. On the other hand, the slower $(T_i)_{i\in \N}$ grows, the less precise we are in \eqref{9/4/17hb}, and by sum the less precise we are on $(\mu_n)_{n\in \N}$. 
 
 This approach is very similar to the chaining method developed by Talagrand \cite{Talagrand}, which, with proper divisions, found many applications  to prove nearly optimal bound of stochastic processes. Recently, the chaining method was also highly developed to study random geometry, and notably of trees constructed by stick-breaking (see e.g. the introduction of my thesis \cite{BRphd} for a detailed explanation on the method) or by aggregation (see e.g. S\'enizergues thesis \cite{SeniThesis}). Since P\'olya urns are closely related to recursive trees it is not surprising that the exact same methods can be used here.  
\subsection{Main ideas of the proof}
First, let us introduce some notations and random variables. To ease the writing, the random variables defined below, unless mentioned otherwise are defined independent and independent of the previously defined random variables.

For $n\in \N$ let $T_n:= \lfloor e^{3n^{1/3}}\rfloor $. Note that $(T_{n+1}-T_{n})/T_{n+1}\sim n^{-2/3}$.  Let $(R_n)_{n\in \N}$ be Bernoulli random variables with parameter $((T_{n+1}-T_n)/T_{n+1} )_{n\in \N}$. Let $(\tilde \Delta_i)_{i\in \N}$ be copies of $\Delta$.  For $n\in \N$, let $Y_n:=R_n\tilde \Delta_n$. For $n\in \N$, let $\Uc_n$ be uniform in $\{1,\dots,n\}$, and let $\Xc_n:=X_{\Uc_n}$.

The following result is our main technical idea to prove Theorem \ref{Thm}. Indeed, it will give us that a.s. given $(X_i)_{i\in \N}$,
\[ \Xc_{T_n}\approxi^{(d)}Y_1+Y_2+\dots+Y_n \]
 and Theorem \ref{Thm} will follow (see Section \ref{sec:Thm}). 
\begin{proposition} \label{Main1} Almost surely $\Xb:=((X_i)_{i\in \N},(U_i)_{i\in \N})$ satisfy the following property: \\
"For every $n\in \N$ large enough we can construct a coupling between $\Xc_{T_n}+Y_n$ and $\Xc_{T_{n+1}}$ such that those two variable are at distance for $\|\cdot \|_\infty$ at least $n^{-2}$ with probability at most $3n^{-4/3}.$"
\end{proposition}
Note that, in the above, the existence of the coupling is random as we see for $n\in \N$, the law of $\Xc_{T_n}+Y_n$ and of $\Xc_{T_{n+1}}$ as random variables on the set of probability measure on $\R^d$. And those laws are determined by $(X_i)_{i\in \N}$. We also consider the $(U_i)_{i\in \N}$ for technical reasons.

To prove Proposition \ref{Main1} we split $\R^d$ into small box of the form, for $i_1,\dots, i_d\in \Z$,
\[ \square^n_{i_1,i_2,\dots, i_d}:= \left [\frac{i_1}{n^2},\frac{i_1+1}{n^2} \right )\times \left [\frac{i_2}{n^2},\frac{i_2+1}{n^2} \right )\times \dots \times \left [\frac{i_d}{n^2},\frac{i_d+1}{n^2} \right ). \] 
And we count given $(X_i)_{1\leq i \leq T_{n}}$ the number of $(X_i)_{i \leq T_{n+1}}$ that are in the small box $\square^n_{i_1,i_2,\dots, i_d}$. Note that this number is exactly equal to 
\begin{equation} T_{n+1}\proba \left ( \left . \Xc_{T_{n+1}} \in \square^n_{i_1,i_2,\dots, i_d}\right  |\Xb\right ). \label{8/4/17h} \end{equation}
Morally, to estimate this number we focus on the $T_n< i\leq T_{n+1}$ such that $U_i\leq T_n$. And for all $T_n< i\neq j \leq T_{n+1}$, given  $U_i,U_j\leq T_n$ and $(X_i)_{1\leq i \leq T_{n}}$, the conditional random variables $X_i,X_j$ are independent and have law  $\Xc_{T_n}+\Delta$. As a result, 
 the number \eqref{8/4/17h} is essentially a sum of independent random variables and so is concentrated, and we have the following approximation,
\[ \proba \left (\left .  \Xc_{T_{n+1}} \in \square^n_{i_1,i_2,\dots, i_d}\right |  \Xb  \right )\approx  \proba \left ( \left . \Xc_{T_n}+Y_n \in \square^n_{i_1,i_2,\dots, i_d} \right | \Xb \right ). \]
Rigorously, we show the next result which directly implies the existence of the desired couplings.
\begin{proposition} \label{Main2} Almost surely as $n\to \infty$
\[ \sum_{i_1,i_2,\dots, i_d\in \Z^d} \left |\proba \left ( \left . \Xc_{T_{n+1}} \in \square^n_{i_1,i_2,\dots, i_d} \right |\Xb  \right )- \proba \left ( \left . \Xc_{T_n}+Y_n \in \square^n_{i_1,i_2,\dots, i_d} \right |\Xb  \right )  \right | \leq 3/n^{4/3}.\]
\end{proposition}
\pagebreak[3]
\section{Proof of Proposition \ref{Main2}}
One can immediately deduce Proposition \ref{Main2}, with a union bound based on the two following lemmas, which will be proved in respectively Section \ref{Sec:Bulk} and Section \ref{Sec:Tail}. 
\begin{lemma}  \label{bulk} For $n\in \N$, let $\diamond_n:= [-n^\gamma,n^\gamma]^d\cap \Z^d$. Almost surely as $n\to \infty$
\[ \sum_{i_1,i_2,\dots, i_d\in \diamond_n} \left |\proba \left ( \left . \Xc_{T_{n+1}} \in \square^n_{i_1,i_2,\dots, i_d} \right |\Xb  \right )- \proba \left ( \left . \Xc_{T_n}+Y_n \in \square^n_{i_1,i_2,\dots, i_d} \right |\Xb  \right )  \right | \leq 2.9/n^{4/3}.\]
\end{lemma} 
\begin{lemma} \label{tail} If $\gamma>0$ is large enough, almost surely as $n\to \infty$
\[  \proba \left ( \left . \| \Xc_{T_{n+1}} \|>n^\gamma  \right |\Xb  \right ) =o(n^{-4/3}) \quad \text{and} \quad \proba \left ( \left . \| \Xc_{T_n}+Y_n \|>n^\gamma \right |\Xb  \right ) =o(n^{-4/3}).\]
\end{lemma} 
\subsection{The bulk : Proof of Lemma \ref{bulk}} \label{Sec:Bulk}
Before proving Lemma \ref{bulk}, let us recall a consequence of Bernstein's inequality tailored for our purpose (see e.g. \cite{Massart} Section 2.8).
\begin{lemma}[Bernstein's inequality] \label{Bernstein} Let $n\in \N$. Let $(X_i)_{1\leq i \leq n}$ be  independent Bernoulli random variables with parameter $(p_i)_{1\leq i \leq n}$. Let $v\geq \sum_{i=1}^n p_i$. Let $S=\sum_{i=1}^n X_i-p_i$. For all $t\geq 0$,
\[ \proba(|S|\geq \sqrt{2vt}+t)\leq e^{-t}. \]
\end{lemma} 
To prove Lemma \ref{bulk}, let us first consider a modification of $\Xc_{T_{n+1}}$. For $n\in \N$, let $(\Uf^n_i)_{i\in \N}$ be a family of independent uniform random variables in $\{1,\dots,T_n\}$. Then for $i\in \N$ let $U_i^n:=U_i$ if $U_i\leq T_n$ and let $U_i^n:=\Uf^n_i$ if $U_i>T_n$. (Morally we resample $U_i$ until we are below $T_n$.) And let 
\begin{equation} \Xf^n_i: = X_i \quad \text{if} \quad  i\leq T_n\quad \text{and let } \quad  \Xf^n_{i}:= X_{U_i^n}+\Delta_{i}  \quad \text{if} \quad i>T_n. \label{22/09/11h} \end{equation}
 Then let $\Xf_{T_{n+1}}:= \Xf_{\Uc_{T_{n+1}}} $. 
The main  interest of this modification is that given $(X_i)_{1\leq i \leq T_n}$, the random variables $(\Xf^n_i)_{T_n<i\leq T_{n+1}}$ are identically distributed and independent. Moreover we have:
\begin{lemma} \label{modif} Almost surely as $n\to \infty$,
\[ \proba(\Xf_{T_{n+1}}\neq \Xc_{T_{n+1}}|\Xb) =(1+o(1))n^{-4/3}. \]
\end{lemma} 
\begin{proof} First, note that, 
\begin{align*} \proba(\Xf_{T_{n+1}}\neq \Xc_{T_{n+1}}|\mathbf X) \leq \frac{1}{T_{n+1}}\sum_{T_n < i\leq T_{n+1}} \proba(\Xf^n_i \neq X_i |\mathbf X) 
 & \leq \frac{1}{T_{n+1}}\sum_{1\leq i\leq T_{n+1}} \proba(U^n_i\neq U_i|\mathbf X)
\\ & = \frac{1}{T_{n+1}}\sum_{T_n<i\leq T_{n+1}} \1_{U_i\in (T_n,T_{n+1} ]}. \end{align*}
The above sum is a sum of independent Bernoulli random variables with parameter bounded by $(T_{n+1}-T_n)/T_{n+1}$. We can thus apply Lemma \ref{Bernstein} with $v=(T_{n+1}-T_n)^2/T_{n+1}$ and $t=2\log(n)$: We get for every $n$ large enough since 
\[ (\sqrt{2vt}+t)/T_{n+1}\sim 2 \sqrt{\log n} (T_{n+1}-T_n)/T_{n+1}^{3/2} \ll n^{-5/3}, \]
that 
\begin{align*} \proba\left (\proba(\Xf_{T_{n+1}}\neq \Xc_{T_{n+1}}|\mathbf X)\geq \frac{v}{T_{n+1}}+n^{-5/3}\right ) & \leq \proba\Big (\frac{1}{T_{n+1}}\sum_{T_n<i\leq T_{n+1}} \1_{U_i\in (T_n,T_{n+1} ]} \geq \frac{v}{T_{n+1}}+n^{-5/3}\Big )
\\ & \leq \proba\Big (\sum_{T_n<i\leq T_{n+1}} \1_{U_i\in (T_n,T_{n+1} ]} -v\geq \sqrt{2vt}+t \Big )  \leq 1/n^2.\end{align*}
The desired result follows from $v/T_{n+1}\sim n^{-4/3}$ and the Borel--Cantelli Lemma.
\end{proof} 
By Lemma \ref{modif} and the triangle inequality, to prove Lemma \ref{bulk}, it suffices to prove the next result:
\begin{lemma}  \label{bulk2} Almost surely as $n\to \infty$
\[ \sum_{i_1,i_2,\dots, i_d\in \diamond_n} \left |\proba \left ( \left . \Xf_{T_{n+1}} \in \square^n_{i_1,i_2,\dots, i_d} \right |\Xb  \right )- \proba \left ( \left . \Xc_{T_n}+Y_n \in \square^n_{i_1,i_2,\dots, i_d} \right |\Xb  \right )  \right | \leq 1.8/n^{4/3}.\]
\end{lemma} 
\begin{proof} Note that conditionally on  $(X_i)_{1\leq i \leq T_n}$, for every $i_1,i_2\dots i_d\in \Z^n$,
\[ T_{n+1}\proba \left ( \left . \Xf_{T_{n+1}} \in \square^n_{i_1,i_2,\dots, i_d} \right |\Xb  \right )=\sum_{i=1}^{T_{n+1}}\1_{\Xf_i^n \in \square^n_{i_1,i_2,\dots, i_d}},\]
is a sum of independent Bernoulli random variables. Furthermore, conditionally on $(X_i)_{1\leq i \leq T_n}$, the random variables $\Xf_{T_{n+1}}$ and $ \Xc_{T_n}+Y_n$ have exactly the same law. As a result, 
\begin{align*} \frac{1}{T_{n+1}}\sum_{i=1}^{T_{n+1}}\proba\left (\left . \Xf_i^n \in \square^n_{i_1,i_2,\dots, i_d} \right |(X_i)_{1\leq i \leq T_n}\right ) & =\proba \left ( \left . \Xf_{T_{n+1}} \in \square^n_{i_1,i_2,\dots, i_d} \right |(X_i)_{1\leq i \leq T_n}  \right ) 
\\ & = \proba \left ( \left . \Xc_{T_n}+Y_n \in \square^n_{i_1,i_2,\dots, i_d} \right |\Xb  \right ) \end{align*}

Thus, we may apply Bernstein's inequality (Lemma \ref{Bernstein})  with
\[ v^n_{i_1,i_2,\dots, i_d}:=T_{n+1}\proba \left ( \left . \Xc_{T_n}+Y_n \in \square^n_{i_1,i_2,\dots, i_d} \right |\Xb  \right ) \]
to estimate 
\[ \delta^n_{i_1,i_2,\dots, i_d}:=\left |\proba \left ( \left . \Xf_{T_{n+1}} \in \square^n_{i_1,i_2,\dots, i_d} \right |\Xb  \right )- \proba \left ( \left . \Xc_{T_n}+Y_n \in \square^n_{i_1,i_2,\dots, i_d} \right |\Xb  \right ) \right |. \]
We choose for  $n\in \N$, $t_n:=\log(n)^2$, and we obtain for every $n\in \N$, $i_1,i_2,\dots, i_d\in \diamond_n$,
\[ \proba \left (T_{n+1}\delta^n_{i_1,i_2,\dots, i_d}>\sqrt{2v^n_{i_1,i_2,\dots, i_d}\log(n)^2}+\log(n)^2 \right )\leq n^{-\log(n)}. \]
This is summable over all $n\in \N$, $i_1,i_2,\dots, i_d\in \diamond_n$, so by the Borel--Cantelli lemma  almost surely for every $n$ large enough, $i_1,i_2,\dots, i_d\in \diamond_n$,
\[ T_{n+1} \delta^n_{i_1,i_2,\dots, i_d} \leq \sqrt{2v^n_{i_1,i_2,\dots, i_d}\log(n)^2}+\log(n)^2 . \]

Summing over all $i_1,i_2,\dots, i_d\in \diamond_n$ we get writing $|\diamond_n|$ for the cardinal of $\diamond_n$, and writing $\delta_n:=\sum_{i_1,i_2,\dots, i_d\in \diamond_n} \delta^n_{i_1,i_2,\dots, i_d}$ for the sum of the lemma,
\[ T_{n+1} \delta_n \leq \sum_{i_1,i_2,\dots, i_d\in \diamond_n} \sqrt{2v^n_{i_1,i_2,\dots, i_d}\log(n)^2}  +|\diamond_n| \log(n)^{2}. \]
Then using the concavity of $x\mapsto \sqrt{x}$,
\[ T_{n+1} \delta_n \leq \sqrt{|\diamond_n|} \left( 2  \sum_{i_1,i_2,\dots, i_d\in \diamond_n}  v^n_{i_1,i_2,\dots, i_d}\log(n)^2 \right )^{1/2}  +|\diamond_n| \log(n)^{2}. \]
Moreover, we have 
\[ \sum_{i_1,i_2,\dots, i_d\in \diamond_n}  v^n_{i_1,i_2,\dots, i_d} = T_{n+1} \sum_{i_1,i_2,\dots, i_d\in \diamond_n} \proba \left ( \left . \Xc_{T_n}+Y_n \in \square^n_{i_1,i_2,\dots, i_d} \right |\Xb  \right )\leq T_{n+1} .\]
Therefore, a.s. for every $n$ large enough,
\[ T_{n+1} \delta_n \leq \sqrt{|\diamond_n|} \sqrt{2T_{n+1}\log(n)^2}+ |\diamond_n| \log(n)^{2}. \]
The desired result follows as $ |\diamond_n|=O(n^{d \gamma})$ and $T_{n+1}=e^{3n^{1/3}(1+o(1))}$.
\end{proof}
\pagebreak[2]
\subsection{The tail : Proof of Lemma \ref{tail}} \label{Sec:Tail}
We now prove Lemma \ref{tail}. By \eqref{eq:tail}, note that if $\gamma>3\beta$ then $\proba(\|Y_n\| >n^\gamma)=O(n^{-2})$ so it suffices to prove the first part of the lemma. Also writing for $n\in \N$,
\[ E_n:= \E \left [\proba \left ( \left . \| \Xc_{T_{n+1}} \|>n^\gamma  \right |\Xb  \right ) \right ], \]
we have by Markov inequality for every $n\in \N$,
\[ \proba\left ( \proba \left ( \left . \| \Xc_{T_{n+1}} \|>n^\gamma  \right |\Xb  \right ) > n^2 E_n \right ) \leq 1/n^2 . \]
So almost surely for every $n\in \N$ large enough, $ \proba \left ( \left . \| \Xc_{T_{n+1}} \|>n^\gamma  \right |\Xb  \right ) \leq n^2 E_n$.

Hence it is enough to upper-bound $E_n$. We have, 
\[ E_n =\E \left [\proba \left ( \left . \| \Xc_{T_{n+1}} \|>n^\gamma  \right |\Xb  \right ) \right ]
 =\E \left [\frac{1}{T_{n+1}} \sum_{i=1}^{T_{n+1}} \1_{\| X_i\| >n^\gamma} \right ]
= \frac{1}{T_{n+1}} \sum_{i=1}^{T_{n+1}} \proba\left (\| X_i\| >n^\gamma\right ).
 \]
 Thus since $\log(T_n)\sim 3n^{1/3}$, to prove Lemma \ref{tail} it suffices to prove the following result.
 \begin{lemma} If $\gamma$ is large enough, as $n\to \infty$,
 \[ \proba\left (\| X_n\| >\log(n)^{\gamma}\right )=o(\log(n)^{-15}). \]
 \end{lemma}
 \begin{proof}
 it is well known that for every $n\in \N$, $X_n$ has the same law as $\sum_{i=1}^n B_i \Delta_i$, where $(B_i)_{i\in \N}$ is a family of Bernoulli random variables with parameter $(1/i)_{i\in \N}$ independent of $(\Delta_i)_{i\in \N}$. With this representation, for $X_n$ to be larger than $\log(n)^{\gamma}$, there must either be at least $\log(n)^2$ term or else one of the term is larger than $\log(n)^{\gamma-2} $. So,
 \[ \proba\left (\| X_n\| >\log(n)^{\gamma}\right )\leq \proba \left (\sum_{i=1}^n B_i>\log(n)^2 \right )+\log(n)^2 \proba \left (\Delta >\log(n)^{\gamma-2} \right ).   \]
 By classical concentration inequalities on sum of independent random variables the first right-hand side term is $o(\log(n)^{-15})$ as $n\to \infty$ (one may for instance use Bernstein's inequality). And by \eqref{eq:tail}, for $\gamma$ large enough,   $\proba \left (\Delta >\log(n)^{\gamma-2} \right )=o(\log(n)^{-17})$.
 \end{proof}
 \section{Proof of Theorem \ref{Thm}} \label{sec:Thm}
 Before proving Theorem \ref{Thm}, let us check two easy results. 
  \begin{lemma} \label{dTV} As $n\sim m\to \infty$, writing $d_{TV}$ for total variation distance, $d_{\TV}(\mu_n/n,\mu_m/m)\to 0$.
 \end{lemma}
 \begin{proof}[Proof of Lemma \ref{dTV}] Assume $n\leq m $. Since for every $1\leq i \leq n$, $\Xc_n=X_i$ with probability $1/n$ and  $\Xc_m=X_i$ with probability $1/m$ we get, $d_{\TV}(\mu_n/n,\mu_m/m)\leq 1-n/m \to 0.$
\end{proof}
 \begin{lemma} \label{cvloi2} As $n\to \infty$, writing $S_n:=Y_1+Y_2+\dots+Y_{n-1}$, $S_n/\log(T_n)^\alpha  \limit^{(d)} \L .$
 \end{lemma}
\begin{proof} Recall that for $n\in \N$, $S_n=^{(d)} \sum_{i=1}^n R_i \Delta_i$, where $(R_i)_{i\in \N}$ are independent Bernoulli random variables of parameter $((T_{n+1}-T_n)/T_{n+1} )_{n\in \N}$. So given \eqref{eq:CVloi}, it suffices to show that $n\to \infty$, $\sum_{i=1}^{n-1} R_i /\log(T_n)$ converge in probability toward $1$. To this end, by the law of large number it suffices to estimate $\E[\sum_{i=1}^{n-1} B_i ]$. We have with elementary calculus
\[ \E\left [\sum_{i=1}^{n-1} R_i \right ]= \sum_{i=1}^{n-1} \frac{T_{n+1}-T_n}{T_{n+1}}\asym_{n\to \infty}  \sum_{i=1}^{n-1} \int_{T_i}^{T_{i+1}} \frac{\mathrm d x}{x} \asym_{n\to \infty} \log(T_n).\]
This concludes the proof. 
\end{proof}
 We first formulate a practical technical consequence of Proposition \ref{Main1}. Let $\mathcal F$ be the set of all functions $f:\R^d\mapsto \R$, 1-Lipschitz for the uniform distance, bounded by 1, and of bounded support. By Proposition \ref{Main1}, a.s. as $n\to \infty$,
 \[ \max_{f\in \mathcal F} \left | \E \left [\left . f(\Xc_{T_{n+1}}) \right |\Xb \right ] - \E\left [ \left . f(\Xc_{T_n}+Y_n) \right |\Xb \right ] \right |=O(1/n^{4/3}).\]
 It follows  by sum that a.s. as $n \to \infty$, $m\in \N$
  \[ \max_{f\in \mathcal F} \left | \E \left [\left . f(\Xc_{T_{n+m}}) \right |\Xb \right ] - \E\left [ \left . f\left (\Xc_{T_n}+\sum_{i=0}^{m-1} Y_{n+i} \right ) \right |\Xb \right ] \right |=O \left (\sum_{i=0}^{m-1} (n+i)^{-4/3} \right )=O(n^{-1/3}),\]
  which we may rewrite into: a.s. as $n\leq m \to \infty$,
    \[ \max_{f\in \mathcal F} \left | \E \left [\left . f(\Xc_{T_m}) \right |\Xb \right ] - \E\left [ \left . f\left (\Xc_{T_m}+S_m-S_n \right ) \right |\Xb \right ] \right |=O(n^{-1/3}).\]

Since for every $m$ large enough, $f\in \mathcal F$, the function $x\mapsto f(x/\log(T_m)^\alpha)$ is also in $\mathcal F$, this implies that a.s. as $n\leq m \to \infty$,
\begin{equation} \max_{f\in \mathcal F} \left | \E \left [\left . f\left (\frac{\Xc_{T_m}}{\log(T_m)^\alpha} \right ) \right |\Xb \right ] - \E\left [ \left . f\left ( \frac{ \Xc_{T_n}+S_m-S_n }{\log(T_m)^\alpha} \right ) \right |\Xb \right ] \right |=O(n^{-1/3}). \label{tempo} \end{equation}
Next, since for any $n\in \N$, almost surely  $|\Xc_{T_n} - S_n|<\infty$, we get from the Lipschitz property, and by the dominated convergence theorem that, a.s.  for every $n\in \N$ as $m\to \infty$,
\[ \E\left [ \left . f\left ( \frac{ \Xc_{T_n}+S_m-S_n }{\log(T_m)^\alpha} \right ) \right |\Xb \right ]- \E\left [ \left . f\left ( \frac{ S_m }{\log(T_m)^\alpha} \right ) \right |\Xb \right ]\limit 0. \]
Thus by \eqref{tempo}, a.s. as $n\to \infty$, 
\[ \limsup_{m\to \infty} \max_{f\in \mathcal F} \left | \E \left [\left . f\left (\frac{\Xc_{T_m} }{\log(T_m)^\alpha} \right ) \right |\Xb \right ] - \E\left [ \left . f\left ( \frac{S_m }{\log(T_m)^\alpha} \right ) \right |\Xb \right ] \right |=O(n^{-1/3}).\]
Since the left-hand side term above does not depends on $n$, and since $S_m$ and $\Xb$ are independent, we get that a.s. for every $f\in \mathcal F$ as $m\to \infty$,
\[ \E \left [\left . f\left (\frac{\Xc_{T_m}}{\log(T_m)^\alpha} \right ) \right |\Xb \right ] - \E\left [f\left ( \frac{S_m }{\log(T_m)^\alpha} \right ) \right ] \limit 0. \]
Moreover by the Portmanteau Theorem and by Lemma \ref{cvloi2} the above second expectation converges toward $\E[f(\Lambda)].$ So must the first: A.s. for every $f\in \mathcal F$ as $m\to \infty$,
\begin{equation} \E \left [\left . f\left (\frac{\Xc_{T_m}}{\log(T_m)^\alpha} \right ) \right |\Xb \right ] \limit \E[\Lambda]. \label{10/4/10h} \end{equation}

 Also as $n\to \infty$, writing $m_n$ for the larger integer with $T_{m_n}\leq n$, we have $T_{m_n}\sim n$ so by Lemma \ref{dTV}, for every $f\in \mathcal F$ as $n\to \infty$, 
\[ \E \left [\left . f\left (\frac{\Xc_{T_{m_n}}}{\log(T_{m_n})^\alpha} \right ) \right |\Xb \right ]-\E \left [\left . f\left (\frac{\Xc_{n}}{\log(T_{m_n})^\alpha} \right ) \right |\Xb \right ] \to 0.\]
Then using that all $f\in \mathcal F$ are Lipschitz and of bounded support, since as $n\to \infty$, $\log(T_{m_n})\to \infty$, and $\log(T_{m_n})\sim \log  (n)$, a.s. for every $f\in \mathcal F$ as $n\to \infty$, 
\[ \E \left [\left . f\left (\frac{\Xc_{T_{m_n}}}{\log(T_{m_n})^\alpha} \right ) \right |\Xb \right ]-\E \left [\left . f\left (\frac{\Xc_{n}}{\log(n)^\alpha} \right ) \right |\Xb \right ] \limit 0.\]
To conclude we finally get that by \eqref{10/4/10h} a.s. for every $f\in \mathcal F$ as $n\to \infty$,
\[ \E \left [\left . f\left (\frac{\Xc_{n}}{\log(n)^\alpha} \right ) \right |\Xb \right ]\limit \E[f(\Lambda)]. \]
Theorem \ref{Thm} follows by the Portmanteau theorem.





\bibliographystyle{unsrt}
\end{document}